\documentclass[12pt,reqno]{amsart}
\usepackage{amsmath, amsthm, amssymb,booktabs}

\advance \topmargin by -\headheight
\advance \topmargin by -\headsep
     
\evensidemargin \oddsidemargin
\newtheorem{theorem}{Theorem}[section]
\newtheorem{lemma}[theorem]{Lemma}

\theoremstyle{definition}

\theoremstyle{remark}

\numberwithin{equation}{section}

\def\calA{{\mathcal A}}

\def\dbN{{\mathbb N}}

\def\dbZ{{\mathbb Z}}

\def\grB{{\mathfrak B}}

\def\grJ{{\mathfrak J}}

\def\grm{{\mathfrak m}}\def\grM{{\mathfrak M}}\def\grN{{\mathfrak N}}
\def\grS{{\mathfrak S}}
\def\grB{{\mathfrak B}}

\def\alp{{\alpha}} 
\def\bet{{\beta}}  
\def\gam{{\gamma}} \def\Gam{{\Gamma}}
\def\del{{\delta}} \def\Del{{\Delta}}

\def\tet{{\theta}}  

\def\lam{{\lambda}}

\def\ome{{\omega}} 
\def\d{{\partial}}
\def\eps{\varepsilon}

\def\le{\leqslant} \def\ge{\geqslant}

\def\d{{\,{\rm d}}}

\begin{document}
\title[Waring's problem for intermediate powers]{On Waring's problem for intermediate 
powers}
\author[Trevor D. Wooley]{Trevor D. Wooley}
\address{School of Mathematics, University of Bristol, University Walk, Clifton, 
Bristol BS8 1TW, United Kingdom}
\email{matdw@bristol.ac.uk}
\date{}
\begin{abstract} Let $G(k)$ denote the least number $s$ such that every sufficiently 
large natural number is the sum of at most $s$ positive integral $k$th powers. We show 
that $G(7)\le 31$, $G(8)\le 39$, $G(9)\le 47$, $G(10)\le 55$, $G(11)\le 63$, 
$G(12)\le 72$, $G(13)\le 81$, $G(14)\le 90$, $G(15)\le 99$, $G(16)\le 108$.
\end{abstract}
\maketitle

\section{Introduction} Conforming to tradition, we denote by $G(k)$ the least number $s$ 
such that every sufficiently large natural number is the sum of at most $s$ positive 
integral $k$th powers. In this note we obtain new bounds for $G(k)$ by exploiting recent 
progress concerning Vinogradov's mean value theorem (see \cite{Woo2012} and 
\cite{BDG2015}).

\begin{theorem}\label{theorem1.1} When $7\le k\le 16$, one has $G(k)\le H(k)$, where 
$H(k)$ is defined by means of Table $1$.
\begin{table}[h]
\begin{center}
\begin{tabular}{cccccccccccccc}
\toprule
$k$ & $7$  & $8$  & $9$ & $10$  & $11$  & $12$  & $13$  & $14$  & $15$  & $16$\\
$H(k)$ & $31$ & $39$ & $47$ & $55$ & $63$ & $72$ & $81$ & $90$ & $99$ & $108$\\
\bottomrule
\end{tabular}\\[6pt]
\end{center}
\caption{Upper bounds for $G(k)$ when $7\le k\le 16$}\label{tab1}
\end{table}
\end{theorem}

For comparison, Vaughan and Wooley \cite{VW1993, VW1995, VW2000} have obtained 
the bounds $G(7)\le 33$, $G(8)\le 42$, $G(9)\le 50$, $G(10)\le 59$, $G(11)\le 67$, 
$G(12)\le 76$, $G(13)\le 84$, $G(14)\le 92$, $G(15)\le 100$, $G(16)\le 109$, in work 
spanning the 1990s. We note in particular that our new bound $G(8)\le 39$ makes 
appreciable progress towards the conjectured conclusion $G(8)=32$ that now seems only 
just beyond our grasp.\par

Our proof of Theorem \ref{theorem1.1} utilises a combination of the powerful estimates 
for mean values restricted to minor arcs recently made available in our work 
\cite{Woo2012} concerning the asymptotic formula in Waring's problem, together with 
the progress on Vinogradov's mean value theorem due to Bourgain, Demeter and Guth 
\cite{BDG2015}. In applications, this mean value estimate has the potential to deliver 
bounds considerably sharper than corresponding pointwise bounds. For intermediate 
values of $k$, these estimates combine with earlier mean value estimates for smooth 
Weyl sums due to Vaughan and the author \cite{VW2000} to deliver satisfactory 
estimates for mixed mean values involving both classical and smooth Weyl sums. This we 
describe in \S3. The corresponding major arc estimates, which we handle in \S4, are 
familiar territory for experts in the subject, and pose no new challenges. For larger values 
of $k$, the relative strength of minor arc estimates available for smooth Weyl sums 
proves superior to our use here of classical Weyl sums, and so no improvements are 
made available for $k\ge 17$.\par

Throughout, the letter $\eps$ will denote a positive number. We adopt the convention that 
whenever $\eps$ appears in a statement, either implicitly or explicitly, we assert that the 
statement holds for each $\eps>0$. In addition, we use $\ll$ and $\gg$ to denote 
Vinogradov's well-known notation, implicit constants depending at most on $k$ and $\eps$, 
as well as other ambient parameters apparent from the context. Finally, we write 
$e(z)$ for $e^{2\pi iz}$, and $[\tet]$ for the greatest integer not exceeding $\tet$.

\section{Preliminaries} Our proof of Theorem \ref{theorem1.1} proceeds by means of the 
circle method. We take the opportunity in this section of outlining our basic approach, 
introducing notation en route that underpins the discussion of subsequent sections. 
Throughout, we let $k$ denote a fixed integer with $7\le k\le 16$. We consider a positive 
number $\eta$ sufficiently small in terms of $k$, and let $n$ be a positive integer 
sufficiently large in terms of both $k$ and $\eta$. Next, write $P=n^{1/k}$, and consider 
positive integers $t$ and $u$ to be fixed in due course. Define the set of 
smooth numbers $\calA_\eta(P)$ by
$$\calA_\eta(P)=\{ n\in [1,P]\cap\dbZ:\text{$p|n$ and $p$ prime $\Rightarrow 
p\le P^\eta$}\}.$$
We consider the number $R(n)$ of 
representations of $n$ in the shape
\begin{equation}\label{2.1}
n=x_1^k+\ldots +x_t^k+y_1^k+\ldots +y_u^k,
\end{equation}
with $1\le x_i\le P$ $(1\le i\le t)$ and $y_j\in \calA_\eta(P)$ $(1\le j\le u)$. We seek to 
show that for appropriate choices of $t$ and $u$, one has $R(n)\gg n^{(t+u)/k-1}$, 
whence in particular $R(n)\ge 1$. Hence, whenever $n$ is a sufficiently large positive 
integer, it follows that $n$ possesses a representation as the sum of at most $t+u$ positive 
integral $k$th powers, whence $G(k)\le t+u$.\par

We define
$$f(\alp)=\sum_{1\le x\le P}e(\alp x^k)\quad \text{and}\quad 
g(\alp)=\sum_{x\in \calA_\eta(P)}e(\alp x^k).$$
When $\grB\subseteq [0,1)$, we put
\begin{equation}\label{2.3}
R(n;\grB)=\int_\grB f(\alp)^tg(\alp)^ue(-n\alp)\d\alp .
\end{equation}
Then it follows from (\ref{2.1}) via orthogonality that $R(n)=R(n;[0,1))$.\par

In order to make further progress, we must define a Hardy-Littlewood dissection of the 
unit interval. Let $\grm$ denote the set of real numbers $\alp\in [0,1)$ satisfying the 
property that, whenever $a\in \dbZ$, $q\in \dbN$, $(a,q)=1$ and
$$|q\alp -a|\le (2k)^{-1}P^{1-k}$$
then one has $q>P$. The set of major arcs $\grM$ corresponding to this set of minor arcs 
$\grm$ is then defined by putting $\grM=[0,1)\setminus \grm$. It is apparent that $\grM$ 
is the union of the intervals
$$\grM(q,a)=\{ \alp\in [0,1):|q\alp -a|\le (2k)^{-1}P^{1-k}\},$$
with $0\le a\le q\le P$ and $(a,q)=1$. \par

In the next section, we establish under appropriate conditions on $t$ and $u$ that one has
$R(n;\grm)=o(P^{t+u-k})$, whilst in \S4 we confirm under the same conditions that 
$R(n;\grM)\gg P^{t+u-k}$. Since $[0,1)=\grM\cup \grm$, these conclusions combine to 
deliver the anticipated lower bound $R(n;[0,1))\gg n^{(t+u)/k-1}$, achieving the goal 
advertised in the opening paragraph of this section.

\section{The minor arc contribution} We now set about establishing that 
$R(n;\grm)=o(P^{t+u-k})$. This we achieve by combining two mean value estimates, the 
first of which concerns classical Weyl sums.

\begin{lemma}\label{lemma3.1}
Whenever $w\ge k(k+1)$, one has
$$\int_\grm|f(\alp)|^w\d\alp \ll P^{w-k-1+\eps}.$$
\end{lemma}

\begin{proof}Denote by $J_{s,k}(X)$ the number of integral solutions of the system of 
equations
$$\sum_{i=1}^s(x_i^j-y_i^j)=0\quad (1\le j\le k),$$
with $1\le x_i,y_i\le X$ $(1\le i\le s)$. Then it follows from \cite[Theorem 2.1]{Woo2012} 
that
\begin{equation}\label{3.1}
\int_\grm |f(\alp)|^{2u}\d\alp \ll P^{\frac{1}{2}k(k-1)-1}(\log P)^{2u+1}J_{u,k}(P).
\end{equation}
However, by reference to \cite[Theorem 1.1]{BDG2015}, we find that whenever 
$2u\ge k(k+1)$, then one has $J_{u,k}(P)\ll P^{2u-k(k+1)/2+\eps}$. The desired 
conclusion follows by substituting this estimate into (\ref{3.1}). 
\end{proof}

We also employ mean value estimates for smooth Weyl sums. We say that the positive real 
number $\lam_{w,k}$ is permissible when, for each $\eps>0$, whenever $\eta$ is a 
sufficiently small positive number, then
\begin{equation}\label{3.2}
\int_0^1|g(\alp)|^{2w}\d\alp \ll P^{\lam_{w,k}+\eps}.
\end{equation}
By reference to the tables of exponents in \S\S9-18 of \cite{VW2000}, we find that the 
exponents $\lam_{w,k}$ and $\lam_{w+1,k}$ recorded in Table 2 are permissible. We are 
at liberty in what follows to assume that $\eta$ has been chosen small enough that the 
estimate (\ref{3.2}) holds for all pairs $(k,w)$ and $(k,w+1)$ occuring in Table $2$.

\begin{table}[h]
\begin{center}
\begin{tabular}{cccccccccccccc}
\toprule
$k$ & $w$  & $\lam_{w,k}$  & $\lam_{w+1,k}$ & $t$ & $u$ & $\del^{-1}$ & $r$& $[U]$\\
\toprule
$7$ & $14$ & $21.1139297$ & $23.0528848$ & $5$ & $26$ & $1267$&17&47\\
$8$ & $18$ & $28.0833353$ & $30.0473193$ & $5$ & $34$ & $1111$&21&58\\
$9$ & $21$ & $33.1033373$ & $35.0727119$ & 7 & 40 & 534&25&86\\
$10$ & $25$ & $40.0895832$ & $42.0677228$ & 9 & 46 & 1792&30&128\\
$11$ & $27$ & $43.1274069$ & $45.1020502$ & 13 & 50 & 2959&34&375\\
$12$ & $32$ & $52.0919461$ & $54.0752481$ & 13 & 59 & 546&38&314\\
$13$ & $36$ & $59.0849135$ & $61.0698015$ & 13 & 68 & 823&42&289\\
$14$ & $40$ & $66.0795485$ & $68.0657585$ & 14 & 76 & 620&46&342\\
$15$ & $44$ & $73.0747403$ & $75.0620643$ & 16 & 83 & 417&50&525\\
$16$ & $47$ & $78.0829008$ & $80.0711728$ & 19 & 89 & 519&55&1780\\
\bottomrule
\end{tabular}\\[6pt]
\end{center}
\caption{Choice of exponents for $7\le k\le 16$}\label{tab2}
\end{table}

We combine these mean value estimates via H\"older's inequality to obtain the bounds 
contained in the following lemma.

\begin{lemma}\label{lemma3.2}
Let $k$, $t$, $u$ and $\del$ be given as in Table $2$. Then one has
$$\int_\grm |f(\alp)^tg(\alp)^u|\d\alp \ll P^{t+u-k-\del}.$$
\end{lemma}

\begin{proof} Let $w$ be given as in Table $2$. Then by H\"older's inequality, the integral 
in question is bounded above by
\begin{equation}\label{3.3}
\biggl( \int_\grm |f(\alp)|^{k(k+1)}\d\alp \biggr)^\ome \biggl( \int_0^1
|g(\alp)|^{2w}\d\alp \biggr)^{\phi_1}\biggl( \int_0^1|g(\alp)|^{2w+2}\d\alp 
\biggr)^{\phi_2},
\end{equation}
where
$$\ome =\frac{t}{k(k+1)},\quad \phi_1=(1-\ome)(w+1)-u/2,\quad 
\phi_2=u/2-(1-\ome )w.$$
Here, in order to verify that this indeed a valid application of H\"older's inequality, it may be 
useful to note that for each value of $k$ in question, one has 
$w=[\tfrac{1}{2}u/(1-\ome)]$.\par

By applying Lemma \ref{lemma3.1} together with (\ref{3.2}) within (\ref{3.3}), we infer 
that
\begin{align}
\int_\grm |f(\alp)^tg(\alp)^u|\d\alp &\ll P^\eps (P^{k(k+1)-k-1})^\ome 
(P^{\lam_{w,k}})^{\phi_1}(P^{\lam_{w+1,k}})^{\phi_2}\notag \\
&\ll P^{t+u-k+\Del+\eps},\label{3.4}
\end{align}
where
$$\Del=\phi_1\Del_w+\phi_2\Del_{w+1}-\ome,$$
in which
$$\Del_v=\lam_{v,k}-2v+k\quad (v=w,w+1).$$
By reference to Table 2, one verifies that whenever $\eps>0$ is sufficiently small, one has 
$\Del<-\del$. The upper bound claimed in the statement of the lemma therefore follows 
for each $k$ in question from (\ref{3.4}).
\end{proof}

An application of the triangle inequality leads from (\ref{2.3}) via Lemma \ref{lemma3.2} 
to the bound
\begin{equation}\label{3.5}
R(n;\grm)=o(P^{t+u-k})
\end{equation}
heralded at the opening of this section.

\section{The major arc contribution and the proof of Theorem \ref{theorem1.1}}
Our goal in this section is the proof of the lower bound $R(n;\grM)\gg P^{t+u-k}$. Experts 
will recognise the argument here to be routine, though not directly accessible from the 
literature. We consequently provide a reasonably complete proof. Our task is made easier 
by the presence of a relatively large number of classical Weyl sums in the integral 
(\ref{2.3}). We require an auxiliary set of major arcs. Let $W=\log \log P$, and define 
$\grN$ to be the union of the intervals
$$\grN(q,a)=\{\alp \in [0,1):|\alp -a/q|\le WP^{-k}\},$$
with $0\le a\le q\le W$ and $(a,q)=1$.\par

We recall from \cite[Lemma 5.1]{Vau1989} that whenever $k\ge 3$ and $s\ge k+2$, one 
has
\begin{equation}\label{4.1}
\int_{\grM\setminus \grN}|f(\alp)|^s\d\alp \ll W^{\eps-1/k}P^{s-k}.
\end{equation}
Moreover, by reference to the tables of \cite[\S\S9-18]{VW2000}, in combination with the 
discussion concluding \cite[\S8]{VW2000} associated with process $D^s$ therein, one finds 
that, with $r$ defined as in Table $2$, one has
\begin{equation}\label{4.2}
\int_0^1|g(\alp)|^{2r}\d\alp \ll P^{2r-k}.
\end{equation}
An application of H\"older's inequality therefore leads from (\ref{2.3}) to the bound
\begin{equation}\label{4.3}
R(n;\grM\setminus \grN) \le \biggl( \int_{\grM\setminus \grN}
|f(\alp)|^{k+4}\d\alp\biggr)^{t/(k+4)}\biggl( \int_0^1 |g(\alp)|^U\d\alp 
\biggr)^{1-t/(k+4)},
\end{equation}
where $U=u/(1-t/(k+4))$. Observe here that for $7\le k\le 16$, it follows from Table $2$ 
that $t<k+4$. Also, a modicum of computation reveals that in each case, one has $U>2r$. 
indeed, there is ample room to spare in the latter inequality, as is evident from Table $2$. 
By importing (\ref{4.1}) and (\ref{4.2}) into (\ref{4.3}), we thus discern that
\begin{align*}
R(n;\grM\setminus \grN)&\ll W^{-t/(k+4)^2}(P^4)^{t/(k+4)}(P^{U-k})^{1-t/(k+4)}\\
&\ll P^{t+u-k}(\log W)^{-1}.
\end{align*}
By combining this estimate with (\ref{3.5}), we may conclude thus far that
\begin{equation}\label{4.4}
R(n)=R(n;\grN)+O(P^{t+u-k}(\log W)^{-1}).
\end{equation}

\par The analysis of the contribution arising from the major arcs $\grN$ is routine. Define
$$S(q,a)=\sum_{r=1}^qe(ar^k/q)\quad \text{and}\quad v(\bet)=\int_0^P
e(\bet \gam^k)\d\gam .$$
Standard arguments (see \cite[Lemma 5.4]{Vau1989} and \cite[Lemma 8.5]{Woo1991}) 
show that there is a positive number $\rho$ having the property that whenever 
$\alp \in \grN(q,a)\subseteq \grN$, one has
$$g(\alp)-\rho q^{-1}S(q,a)v(\alp-a/q)\ll P(\log P)^{-1/2}.$$
Under the same conditions, the relation
$$f(\alp)-q^{-1}S(q,a)v(\alp-a/q)\ll \log P$$
is immediate from \cite[Theorem 4.1]{Vau1997}. Thus we find that when 
$\alp\in \grN(q,a)\subseteq \grN$, one has
$$f(\alp)^tg(\alp)^u-\rho^u\left(q^{-1}S(q,a)v(\alp-a/q)\right)^{t+u}\ll 
P^{t+u}(\log P)^{-1/2}.$$
Integrating over $\grN$, we infer that
\begin{equation}\label{4.5}
\int_\grN f(\alp)^tg(\alp)^ue(-n\alp)\d\alp =\rho^u\grS(n;W)\grJ(n;W)+
O(P^{t+u-k}(\log P)^{-1/3}),
\end{equation}
where
$$\grS(n;W)=\sum_{1\le q\le W}\sum^q_{\substack{a=1\\ (a,q)=1}}
\left( q^{-1}S(q,a)\right)^{t+u}e(-na/q)$$
and
$$\grJ(n;W)=\int_{-WP^{-k}}^{WP^{-k}}v(\bet)^{t+u}e(-\bet n)\d\bet .$$
A comparison with classical singular series and integrals conveys us from here, via 
\cite[Chapter 4]{Vau1997}, for example, to the relations
$$\grS(n;W)=\grS(n)+o(1)$$
and
$$\grJ(n;W)=\frac{\Gam (1+1/k)^{t+u}}
{\Gam((t+u)/k)}n^{(t+u)/k-1}+o(n^{(t+u)/k-1}),$$
in which
$$\grS(n)=\sum_{q=1}^\infty \sum^q_{\substack{a=1\\ (a,q)=1}}
\left( q^{-1}S(q,a)\right)^{t+u}e(-na/q)$$
is the conventional singular series associated with Waring's problem for sums of $t+u$ 
integral $k$th powers.\par

\par Substituting these expressions into (\ref{4.5}), and from there into (\ref{4.4}), we 
conclude that
$$R(n)=\rho^u\grS(n)\frac{\Gam(1+1/k)^{t+u}}{\Gam((t+u)/k)}n^{(t+u)/k-1}+
o(n^{(t+u)/k-1}).$$
Here, we have made use of the fact that since $t+u\ge 4k$ in each case under 
consideration, the standard theory of the singular series (see 
\cite[Theorems 4.3 and 4.5]{Vau1997}) suffices to confirm that $1\ll \grS(n)\ll 1$. In 
particular, one has $R(n)\gg n^{(t+u)/k-1}$. As discussed earlier, this establishes that 
$G(k)\le t+u$, with $t$ and $u$ determined via Table $2$, and thus the proof of Theorem 
\ref{theorem1.1} is complete. 

\providecommand{\bysame}{\leavevmode\hbox to3em{\hrulefill}\thinspace}

\end{document}